\documentclass[a4paper,11pt]{amsart}
\usepackage{hyperref,latexsym}
\usepackage{enumerate}

\theoremstyle{plain}
\newtheorem{theorem}{Theorem}[section]
\newtheorem{lemma}[theorem]{Lemma}

\theoremstyle{definition}

\theoremstyle{remark}
\newtheorem{remark}{Remark}

\begin{document}

\title[Invisible directions]
      {A remark on the number of invisible directions for a smooth Riemannian metric}

\date{23 April 2013}
\author{Misha Bialy}
\address{School of Mathematical Sciences, Raymond and Beverly Sackler Faculty of Exact Sciences, Tel Aviv University,
Israel} \email{bialy@post.tau.ac.il}
\thanks{Partially supported by ISF grant 128/10}

\subjclass[2010]{ } 

\begin{abstract}In this note we give a construction of a smooth Riemannian
metric on $\mathbf{R}^n$ which is standard Euclidean outside a
compact set $K$ and such that it has $N={n(n+1)}/{2}$ invisible
directions, meaning that all geodesics lines passing through the
set $K$ in these directions remain the same straight lines on exit. For example
in the plane our construction gives three invisible directions.
This is in contrast with billiard type obstacles where a very
sophisticated example due to A.Plakhov and V.Roshchina gives $2$
invisible directions in the plane and $3$ in the space.

We use reflection group of the root system $A_n$ in order to make the directions of the
roots invisible.
\end{abstract}

\maketitle

\section{The Problem of invisibility}
\label{sec:intro} Consider a smooth Riemannian metric $g$ on
$\mathbf{R}^n$ which is supposed to be standard Euclidean outside a
compact set $K$. Geodesics of the metric outside the set $K$ are
straight lines and are deformed somehow inside $K$. Following \cite{P}, we
say that the obstacle $K$ is invisible in the direction $v$ if every
geodesic in the direction $v$ passing the obstacle remains the same
straight line. This direction $v$ is called the direction of
invisibility in this case. It is important question how many
invisible direction can exist for a non-flat smooth Riemannian metric. It was shown
in \cite{Croke} basing on \cite{Gromov} and generalizing previous
results \cite{G-G},\cite{Mi} that the invisibility in all
directions implies that the metric is isometric to Euclidean one. This
is the so called lens rigidity phenomena (see also \cite{S-U} for further developments).
It is a natural question to ask how large the set of
invisible direction can be. In particular can it be large or even
infinite. In \cite{P} this question is studied for an analogous model of perfect reflections. A very sophisticated construction of two invisible
directions in the plane and three in the space is given in \cite{P}.
On the other hand there are non-smooth examples of Riemannian metrics with singularities 
which are perfect lenses (see \cite{LP} for further references).

Our remark is that for the smooth case one can construct a Riemannian metric with $N=n(n+1)/2$
invisible directions in $\mathbf{R}^n$.

\begin{theorem}
\label{main} There exists a family of smooth non-flat Riemannian metrics
$g$ on $\mathbf{R}^n$ which are Euclidean outside a compact set $K$
and having $N=n(n+1)/2$ invisible directions.

\end{theorem}
The idea is that $N$ is the number of components of the metric
tensor and also is the number of positive roots of the root system
$A_n$. It is not clear if other reflection groups can be used in a similar manner.

\begin{remark}A similar problem for conformaly flat metrics can be posed also.
 Our construction in this case gives only one invisible direction. Analogously,
 one can construct  metrics in the diagonal form with $n$ invisible directions.
 \end{remark}

\begin{remark} It is worth mentioning  that there exist smooth Finsler non-flat metrics
which are compactly supported and
have all the directions invisible. This is a very well known observation related to Hopf rigidity. These metrics can be constructed by the action of small compactly supported symplectic diffeomorhism of $T^*\mathbf{R}^n$ on the Lagrangian foliation corresponding to Minkowskii metric.
\end{remark}

\begin{remark}One can use this construction on other
manifolds also. The most natural is to implant such a metric into a
small ball of flat torus.
Then the resulting geodesic flow has $N=n(n+1)/2$ invariant Lagrangian torii.
It would be interesting to understand
geometry and dynamics of these examples further.
\end{remark}

\section{Using generating functions}
In this section we use generating functions to create Lagrangian
submanifolds in the energy level of the Riemannian metric in
question.

Recall that the root system $A_n$ can be realized as a set of the
integer vectors  $e_i-e_j$ (of the length $\sqrt 2$) in
$\mathbf{R}^{n+1}$, where $\mathbf{R}^n$ is viewed as hyperplane of
$\mathbf{R}^{n+1}$ defined as $\{x_1+...+x_{n+1}=0\}$.

For our purposes it
will be convenient to arrange the roots in the following order. Let
$v_1,..,v_N$ are such that first $n$ are defined by
$v_i=e_i-e_{n+1}$ and the rest are $e_i-e_j$ for $1\leq i<j\leq n$.
So there are $N=n(n+1)/2$ of them and together with their negatives
they form all the roots. Notice that $(v_1,..,v_n)$ form a basis of
$\mathbf{R}^n$ and the rest are are their differences $v_i-v_j,\
1\leq i<j\leq n$.

For every $i=1,..,N$ consider the Lagrangian sections $L_i$ of
$T^*\mathbf{R}^n$ equipped with the standard symplectic structure
which are defined by the generating function
\begin{equation}
\label{Lagrangian} S_i(x)=(v_i,x)+\epsilon\phi_i(x),\quad
L_i=\{p=\nabla S_i=v_i+\epsilon\nabla\phi_i\},
 \end{equation} where $\phi_i$ are any smooth
functions on $\mathbf{R}^n$ with the support in a ball $B$. Here and
later we denote by  $(,)$ the standard scalar product. It is a well
known fact that any root system determines the scalar product
uniquely. Therefore we have the following

\begin {theorem}
\label{Main}
If $\epsilon>0$ is small enough then  there exists and unique
Riemannian metric $g$ on $\mathbf{R}^n$ such that all $L_i$ lie in
the level $\{h=1\}$ of the corresponding Hamiltonian  function $h$.
Moreover this metric is standard Euclidean outside $B$.
\end{theorem}
\begin{proof}
Let the Riemannian metric and the Hamiltonian function are given by the
matrices $G$ and $H$, $G=H^{-1}$:
 $$g=\sum_{i,j=1}^n g_{ij}
dx_idx_j=(Gdx,dx)\  ; \ \\h=\frac{1}{2} \sum_{i,j=1}^n
h_{ij}p_ip_j=\frac{1}{2}(Hp,p).$$

For a given choice of the functions $\phi_i, i=1,..,N$, the
requirements $L_i\in\{h=1\}$ form a linear system of $N$ inhomogeneous
equations on the $N$ unknown coefficients $h_{ij}$. For
$\epsilon=0$ this system reads simply that the vectors $v_i$ have
all length $\sqrt 2$, which has unique solution namely the standard
Euclidean metric, i.e $G=H=Id$. Therefore the determinant of the
system is not zero and then for $\epsilon $ small enough there is a unique
solution also which is a positive definite form.
\end{proof}
\begin{remark}
In principle one could find the solution $h_{ij}$ of this linear
system explicitly in terms of derivatives of the functions $\phi_i$,
thus determining the metric coefficients (see also Section 4).
\end{remark}
Since by Theorem \ref{Main} all Lagrangian submanifolds $L_i$ lie in an energy level of $h$, then it follows that $L_i$ are invariant under the geodesic flow. Thus every geodesic straight line of
the metric $g$ which enters the ball $B$ in one of the directions
$v_i,i=1,..,N$ leaves the ball by a parallel straight line. Next we can use a
simple symmetry idea which makes this line to be identical with the
initial one.
\section{Symmetrizing the metric}
The symmetrization procedure is based on the folloing obvious
\begin{lemma}Suppose a Riemannian metric $g$ is invariant under
the reflection $s_v$, which is the reflection with respect to the hyperplane $P_v$ orthogonal to
$v$ in $\mathbf{R}^n$. Then any geodesic which crosses $P_v$
orthogonally is symmetric with respect to $P_v$.
\end{lemma}
As a corollary we have the following. Take the metric $g$
constructed in a previous section, where the ball $B$ lies in one
halfspace with respect to $P_v$, and reflect the metric to the other
halfspace of $P_v$. By the lemma one gets a new metric supported on
$B\cup s_v (B)$ with the property that the direction $v$ is not
visible.

 Using this observation we proceed as follows. Consider the Weyl group $W$
 of the root system $A_n$ generated by the reflections $s_{v_i},\
 i=1,..,N$. Consider an arbitrary point $P_1$ lying inside the Weyl chamber $C$ together with
 a sufficiently small ball $B_1$ centered in $P_1$.

 Use the construction (\ref {Lagrangian}) to define Riemannian metric $g_1$ on it. The Weyl group $W$ acts on the chambers simply transitively.
 We define the points $P_i$
 and the balls $B_i$ together with the Riemannian metric on $B_i$ pushed forward from the initial one. Here $i$ ranges from $1$ to $|W|=(n+1)!$.

 I claim that so constructed metric $g$ on $\mathbf{R}^n$ is invisible in the directions
 of every root $v_k$ of $A_n$. Indeed, by the construction every reflection $s_{v_k}$ is an isometry
 of the constructed metric $g$. Moreover by formula (\ref {Lagrangian}) any geodesic straight line
 passing every ball $B_i$ in the direction $v_k$ remains a parallel straight line and so crosses $P_{v_k}$
 orthogonally. Therefore by the lemma the whole geodesic is symmetric and so the direction ${v_k}$
 is invisible.

 Moreover if the radius of the initial ball $B_1$ was chosen sufficiently small
 then every geodesic in the direction $v_k$ crosses in fact only two of the balls or non,
 where these two are symmetric with respect to reflection $s_{v_k}$. Indeed,
 let us arrange all the balls into symmetric pairs with respect to $P_{v_k}$. Since the Weyl group
 acts by orthogonal transformations on $\mathbf{R}^n$, so no three centers $P_i$ of the balls
 can lie on a straight line. Then it is obvious that if the radii of the balls are small enough, then the convex hulls of these pairs are all disjoint.

 \section{Checking non-flatness}
 In this section we check that the the constructed Riemannian
 metrics are in fact non-flat. Of course it is imposable to compute
 curvatures in finite time. Therefore we proceed by a different argument.
 Suppose on the contrary that the metric $g$ is flat. In such a case
 it must be isometric to Euclidean (see for example \cite {Croke2})
 and therefore scalar product with respect to $g$ of the geodesic fields
 given by the sections $L_i$ and $L_j$ must be constant on
 $\mathbf{R}^n$ for all $i,j=1,..,N$. In particular, we can write
 the following identities.
\begin{align*}
(H(v_k+\epsilon\nabla\phi_k),\ v_k+\epsilon\nabla\phi_k)=2;\\
(H(v_k-v_l+\epsilon\nabla\phi_{kl}),\ v_k+\epsilon\nabla\phi_k)=const;\\
(H(v_k-v_l+\epsilon\nabla\phi_{kl}),\ v_l+\epsilon\nabla\phi_l)=const;\\
(H(v_l+\epsilon\nabla\phi_l),\ v_l+\epsilon\nabla\phi_l)=2.
\end{align*}
Here $1\leq k<l\leq n$ and the function $\phi_{kl}$ corresponds to
the root $v_k-v_l$ in the formula (\ref{Lagrangian}). Let $H=Id
+\epsilon H_1+...$ and extract in these equations terms of order
$\epsilon$. We have
\begin{align}
 (H_1(v_k),v_k)+2(v_k, \nabla\phi_k)=0;\\
(H_1(v_k-v_l),v_k)+(\nabla\phi_{kl},\ v_k)+(\nabla\phi_k,v_k-v_l)=const;\\
(H_1(v_k-v_l),v_l)+(\nabla\phi_{kl},\ v_l)+(\nabla\phi_l,v_k-v_l)=const;\\
(H_1(v_l),v_l)+2(v_l, \nabla\phi_l)=0.
\end{align}
Subtract (2) from (3) and also add (4) and (5):
\begin{align}
-(H_1(v_l),v_k)+(\nabla\phi_{kl},\
v_k)-(\nabla\phi_k,v_k+v_l)=const;\\
(H_1(v_k),v_l)+(\nabla\phi_{kl},\ v_l)+(\nabla\phi_l,v_k+v_l)=const
\end{align}
Since $H_1$ is symmetric matrix we can add the last two equations to
get.
\begin{equation}
(\nabla(\phi_{kl}-(\phi_k-\phi_l)),v_k+v_l)=const.
\end{equation}
Outside the support $B$ the LHS of (8) is obviously zero, so
\begin{equation}
(\nabla(\phi_{kl}-(\phi_k-\phi_l)),v_k+v_l)=0.
\end{equation}
But these are strong restrictions on the functions $\phi_i$'s of the
construction. Thus if we choose functions $\phi_k, \phi_l, \phi_{kl}$ in (\ref{Lagrangian}) violating at least one of these identities then the corresponding metric is not flat.

\section*{Acknowledgements}
It is a pleasure to thank S. Tabachnikov
and A. Plakhov for useful discussions during the programm "Topology in Dynamics and Physics"
in Tel Aviv University.

\end{document}